\documentclass[12pt]{article}
\usepackage[utf8]{inputenc}
\usepackage[T1]{fontenc}
\usepackage{amsmath}
\usepackage{amsfonts}
\usepackage{amssymb}
\usepackage{amsthm}
\usepackage{stmaryrd}
\usepackage{graphicx}
\usepackage{tikz}
\usepackage{tikz-cd}
\usepackage{xcolor}
\usepackage{nicefrac}
\usepackage{hyperref}

\newtheorem{theorem}{Theorem}[section]

\newtheorem{lemma}[theorem]{Lemma}

\theoremstyle{definition}

\newtheorem{remark}[theorem]{Remark}

\newcommand{\Q}{\mathbb{Q}}
\newcommand{\F}{\mathbb{F}}
\newcommand{\Fpt}{\F_{p}(\!(t)\!)}
\newcommand{\Fptt}{\F_{p}\llbracket t\rrbracket}

\title{On the nilpotency of locally pro-$p$ contraction groups}
\author{Alonso Beaumont}
\date{January 2025}

\begin{document}

\maketitle

\begin{abstract}
    H. Glöckner and G. A. Willis have recently shown \cite{GW21} that locally pro-$p$ contraction groups are nilpotent. The proof hinges on a fixed-point result \cite[Theorem B]{GW21}: if the local field $\Fpt$ acts on its $d$-th power $\Fpt^{d}$ additively, continuously, and in an appropriately equivariant manner, then the action has a non-zero fixed point. We provide a short proof of this theorem.
\end{abstract}

\section{Introduction}

A locally compact topological group $G$ is said to be a \textit{contraction group} if it has a continuous automorphism $\alpha$ whose forward orbits converge to the identity element $e$:
\[\forall g\in G, \quad \alpha^{\circ n}(g)\underset{n\rightarrow\infty}{\longrightarrow} e.\]
This dynamical property has interesting connexions with the purely group-theoretic notion of nilpotency. This relationship is best illustrated in the case where $G$ is a Lie group: a contracting automorphism on $G$ provides a positive gradation for its Lie algebra $\mathfrak{g}$, proving in particular that $\mathfrak{g}$ is nilpotent (see \cite{S86}). A thorough study of totally disconnected locally compact (tdlc) contraction groups is carried out in \cite{GW10}. One of its main structural results asserts that any tdlc contraction group is a direct product of a torsion-free contraction group, which is always nilpotent, and a torsion contraction group \cite[Theorem B]{GW10}. The remaining questions concerning the torsion part were treated in \cite{GW21}. In order to accommodate for torsion, an additional hypothesis is needed: a tdlc group is called \textit{locally pro-$p$} if it contains an open pro-$p$ subgroup. For example, analytic groups over the field $\Q_{p}$ of $p$-adic numbers, or the field $\Fpt$ of Laurent series with coefficients in $\F_{p}$, are locally pro-$p$. In \cite{GW21} it is shown that locally pro-$p$ contraction groups are nilpotent. The major part of the proof consists in establishing the following fixed point result:

\begin{theorem}[{\cite[Theorem B]{GW21}}]
    \label{theo}
    Let $d\geq1$ and consider a continuous action $(\Fpt,+)\curvearrowright (\Fpt^{d},+)$ by group automorphisms. If
    \[\forall x\in\Fpt,\;\forall u\in \Fpt^{d},\; (tx)\cdot (tu)=t(x\cdot u)\]
    then the action has a non-zero fixed point.
\end{theorem}

Glöckner and Willis's proof of this result involves representing endomorphisms of the abelian group $\Fpt^d$ as infinite block matrices satisfying certain conditions. We refer the reader to Section 3 of \cite{GW21} for details. The purpose of this note is to provide an alternative, shorter proof of this result. It has two main ingredients: a study of $\F_{p}$-representations of elementary abelian $p$-groups, and an application of Ascoli's theorem to the automorphism group $\mathrm{Aut}(\Fpt^{d})$.

\paragraph{Acknowledgements.} I'd like to thank David J. Benson for his contribution to Lemma \ref{lem}, as well as Adrien Le Boudec and Helge Glöckner for their helpful comments.

\section{Proof of the Theorem}

We begin with a lemma. A linear representation $V$ of a group $G$ is said to be \textit{locally finite} if its finitely generated subrepresentations are finite-dimensional. We denote by $V^{G}$ the set of $G$-fixed points in $V$.

\begin{lemma}
    \label{lem}
    Let $G$ be an elementary abelian $p$-group, and $V$ a locally finite representation of $G$ over $\F_{p}$. If $V$ is infinite-dimensional, then either
    \[V^{G} \quad or \quad \left(V/V^{G}\right)^{G}\]
    is infinite-dimensional.
\end{lemma}

\begin{proof}
    First, suppose that $G$ is finite. Then $\F_{p}[G]$ is a finite-dimensional $\F_{p}$-algebra: if the $\F_{p}[G]$-module $V^{G}$ were finite then its essential extension $V$ would also be finite (see for instance \cite{RZ58}), so $V^{G}$ is infinite-dimensional. In order to give a self-contained presentation, we will provide a more direct proof. Write $G=\langle g_{1},\cdots, g_{r}\rangle$. By the local finiteness of $V$, there is a strictly increasing family $(V_{n})_{n\geq0}$ of finite subrepresentations. For all $n\geq0$, write $d_{n}=\dim V_{n}$ and consider the sequence
    \[\forall i\in\{0,\cdots, p\},\quad d_{n}(i)=\dim\left(\ker(g_{1}-\mathrm{id})^{i}\cap V_{n}\right).\]
    By a standard application of the rank-nullity theorem, $(d_{n}(i+1)-d_{n}(i))_{0\leq i< p}$ is a decreasing sequence. Moreover, $d_{n}(p)=d_{n}$ since $(g_{1}-\mathrm{id})^{p}=g_{1}^{p}-\mathrm{id}=0$. We deduce that $d_{n}(1)=\dim V_{n}^{\langle g_{1}\rangle}\geq d_{n}/p$. The elements $g_{1}$ and $g_{2}$ commute, so we can consider the action of $\langle g_{2}\rangle$ on $V^{\langle g_{1}\rangle}$. By repeating the argument above we obtain $\dim V_{n}^{\langle g_{1},g_{2}\rangle}\geq d_{n}/p^{2}$ and by a finite induction,
    \[\dim V^{G}\geq\dim V_{n}^{G}\geq d_{n}/p^{r}\underset{n\rightarrow+\infty}{\longrightarrow}+\infty.\]

    Now suppose that $G$ is infinite, and that both $V^{G}$ and $(V/V^{G})^{G}$ are finite-dimensional. Then the preimage $V'$ of $(V/V^{G})^{G}$ in $V$ is finite-dimensional: its $G$-action factors through a finite quotient. Since $G$ is elementary abelian, this quotient is a direct summand. We can therefore write $G=G_{0}\times G_{1}$, where $G_{0}$ is finite and $G_{1}$ acts trivially on $V'$. In particular, $G_{1}$ acts trivially on $V'\cap V^{G_{0}}$ so $V'\cap V^{G_{0}}\subset (V^{G_{0}})^{G_{1}}=V^{G}$. Hence,
    \[\left(V'\cap V^{G_{0}}\right)/V^{G}=\left(V^{G_{0}}/V^{G}\right)^{G}=0.\]
    Suppose $V^{G_{0}}/V^{G}$ is non-trivial. Since $V^{G_{0}}/V^{G}$ is locally finite, it would have a non-trivial finite-dimensional subrepresentation, and hence a non-zero $G$-fixed point, because $G$ acts as an abelian unipotent group. We deduce that
    $V^{G_{0}}/V^{G}=0$. This is absurd since $G_{0}$ is finite and thus $V^{G_{0}}$ is infinite-dimensional.
\end{proof}

Let $d\geq1$. We consider $\Fpt^{d}$ as an additive tdlc group, equipped with a continuous automorphism $t\cdot\mathrm{id}:u=(u_{1},\cdots,u_{d})\mapsto tu=(tu_{1},\cdots,tu_{d})$. The group $\mathrm{Aut}(\Fpt^{d})$ of continuous automorphisms is endowed with the compact-open topology. In particular, a continuous action $G\curvearrowright \Fpt^{d}$ by group automorphisms induces a continuous morphism $\varphi:G\rightarrow\mathrm{Aut}(\Fpt^{d})$. We restate Theorem \ref{theo} in the following manner:

\begin{theorem}
    \label{theor}
    A continuous morphism $\varphi:\Fpt\rightarrow\mathrm{Aut}(\Fpt^{d})$ satisfying
    \begin{equation}
    \label{1}
        \forall x\in\Fpt,\;\forall u\in \Fpt^{d},\;\varphi(tx)(u)=t\varphi(x)(t^{-1}u)
    \end{equation}
    has a non-zero fixed point: there exists $u\in\Fpt^{d}\backslash\{0\}$ such that $\varphi(x)(u)=u$ for all $x\in\Fpt$.
\end{theorem}

\begin{proof}
    The proof will be done in two steps. We will first construct a nested family of invariant subgroups, and then we will use it to locate orbits of increasingly small diameter. We will use the following notation: $G=\varphi(\Fpt)$, $G_{\ell}=\varphi(t^{-\ell}\Fptt)$ for all $\ell\geq0$, $B=\Fptt^{d}$, and $S=B\backslash tB$.

    \paragraph{Step 1.}  Let $M$ be the unique maximal $G$-invariant subgroup contained in $B$. It can be defined as the group generated by all $G$-invariant subgroups in $B$. We will prove the following: $M$ is non-trivial, and $A:=\bigcup_{n\geq0}t^{-n}M$ is a $G$-invariant $\Fpt$-subspace of $\Fpt^{d}$ in which $M$ is open.

    Since $\varphi$ is continuous, $G_{\ell}$ is compact, and so by Ascoli's theorem (\cite[Theorem 47.1]{M14}) it is an equicontinous family of automorphisms of $\Fpt^{d}$. In particular, there exists a neighbourhood of $0$ whose $G_{\ell}$-orbit is contained in $B$. The subgroup generated by this orbit is therefore open, $G_{\ell}$-invariant, and contained in $B$. Let $M_{\ell}$ be the largest $G_{\ell}$-invariant subgroup in $B$. By the previous argument $M_{\ell}$ contains an open subgroup and thus it is itself open. By property (\ref{1}), we have
    \[\forall x\in t^{-\ell}\Fptt,\quad \varphi(x)(t^{-1}M_{\ell})=t^{-1}\varphi(tx)(M_{\ell})=t^{-1}M_{\ell}\]
    and so $t^{-1}M_{\ell}$ is also $G_{\ell}$-invariant. This means that $M_{\ell}$ and $S$ have a non-empty intersection, for otherwise the subgroup $M_{\ell}+t^{-1}M_{\ell}$, which is $G_{\ell}$-invariant and strictly larger that $M_{\ell}$, would be contained in $B$. Therefore, the decreasing sequence $(M_{\ell}\cap S)_{\ell\geq0}$ consists of non-empty compact sets, and hence its intersection is non-empty. We have proven that
    \[M:=\bigcap_{\ell\geq0}M_{\ell}\]
    is a non-trivial subgroup of $\Fpt^{d}$. It is in fact the largest $G$-invariant subgroup contained in $B$, due to the maximality of each $M_{\ell}$. Again using (\ref{1}), we see that $tM$ is $G$-invariant and so $tM\subset M$. As an intersection of closed subgroups, $M$ is closed. We deduce that $M$ is an $\Fptt$-module and thus
    \[A:=\bigcup_{n\geq0}t^{-n}M\]
    is a $G$-invariant $\Fpt$-subspace of $\Fpt^{d}$. In particular, $A$ is closed in $\Fpt^{d}$ and therefore it is a Baire space. As such, some $t^{-n}M$, or equivalently all of them, are open in $A$.

    \paragraph{Step 2.} For all $n\geq0$, let $V_{n}:=t^{-n}M/M$. This quotient is an $\F_{p}$-vector space of finite dimension endowed with a linear $G$-action. This means that $V:=A/M=\bigcup_{n\geq0}V_{n}$ satisfies the hypotheses of Lemma \ref{lem}. Let us show that $V$ has infinitely many $G$-fixed points. Suppose otherwise: $V^{G}$ would then be contained in $V_{n}$ for some $n\geq0$. On the other hand, by Lemma \ref{lem} $(V/V^{G})^{G}$ would be infinite: since $V/V^{G}$ surjects $G$-equivariantly onto $V/V_{n}$ with finite kernel, $(V/V_{n})^{G}$ would also be infinite. We claim this is absurd because $V^{G}$ and $(V/V_{n})^{G}$ are in bijection with one another. Indeed, we way choose a section $V/V_{n}=A/t^{-n}M\rightarrow A$ and compose it with $t^{n}\cdot\mathrm{id}:A\rightarrow A$ followed by $A\rightarrow V$. This defines an isomorphism $V/V_{n}\rightarrow V$ which by property (\ref{1}) preserves $G$-fixed points.

    We have shown that $V^{G}$ is infinite. In particular, for all $n\geq0$, $V^{G}$ is not contained in $V_{n}$: we can find some $\overline{v_{n}}\in (A/M)^{G}$ that lifts to $v_{n}\in A\backslash t^{-n}M$. Let $k_{n}> n$ be the unique integer such that $w_{n}:=t^{k_{n}}v_{n}\in M\backslash tM$. For all $x\in \Fpt$, we have
    \[\varphi(x)(w_{n})-w_{n}=t^{k_{n}}\left(\varphi(t^{-k_{n}}x)(v_{n})-v_{n}\right)\in t^{k_{n}}M.\]
    Since $M\backslash tM$ is compact, the family $(w_{n})_{n\geq0}$ has an accumulation point $w$. For all $x\in\Fpt$ and infinitely many $n$, we have $\varphi(x)(w)-w\in t^{k_{n}}M$, so $\varphi(x)(w)=w$: we have constructed a non-zero $G$-fixed point in $A$.
\end{proof}

\begin{remark}
    Glöckner and Willis state Theorem \ref{theor} in a slightly different way: they endow $\mathrm{Aut}(\Fpt^{d})$ with the finer Braconnier topology, and assume that $\varphi$ is continuous with respect to it. This assumption is in fact no stronger than ours (see \cite[Lemma 9.13]{S06}). In light of this, Step 1 of the proof above may be started as follows. The subgroup $\{g\in G_{\ell}\;|\; g\cdot B=B\}$ is open in $G_{\ell}$ for the Braconnier topology, and therefore it is of finite index. Thus, the intersection $\bigcap_{g\in G_{\ell}}g\cdot B$ is finite, yielding an open $G_{\ell}$-invariant subgroup in $B$.
\end{remark}

\begin{remark}
    Let $(G_{1},\alpha_{1})$ and $(G_{2},\alpha_{2})$ be two tdlc contraction groups, equipped with their respective contracting automorphisms. A continuous action $G_{1}\curvearrowright G_{2}$ may be called equivariant if $\alpha_{1}(x)\cdot\alpha_{2}(y)=\alpha_{2}(x\cdot y)$ for all $x\in G_{1}$, $y\in G_{2}$. Then the argument of Step 1 can be applied in this more general setting: there exists a compact, non-trivial subgroup $H$ of $G_{2}$ such that $\alpha_{2}(H)\subset H$ and $G_{1}\cdot H=H$.
\end{remark}

\bibliographystyle{amsalpha}
\bibliography{bib}

\end{document}